\documentclass[halfparskip]{article}
\usepackage[latin1]{inputenc}
\usepackage{latexsym}
\usepackage{graphicx}
\usepackage{euscript}
\usepackage{amsfonts}
\usepackage{amssymb}
\usepackage{amsthm}
\usepackage{amsmath}
\usepackage{enumerate}

\newtheorem{theorem}{Theorem}

\newtheorem{lemma}{Lemma}

\newtheorem{remark}{Remark}

\newcommand{\R}{\mathbb{R}}
\newcommand{\C}{\mathbb{C}}
\newcommand{\N}{\mathbb{N}}

\newcommand{\Q}{\mathbb{Q}}

\newcommand{\Exp}{\textnormal{Exp}}

\newcommand{\Bo}{{\cal B}}
\newcommand{\TBo}{{\tilde{\cal B}}}

\newcommand{\HC}{{\cal HC}}
\newcommand{\FHC}{{\cal FHC}}

\newcommand{\Imm}{\textnormal{Im}}
\newcommand{\Reee}{\textnormal{Re}}

\newcommand{\ldens}{\underline{\textnormal{dens}}}

\newcommand{\linspan}{\mathrm{linspan}}

\title{\bf \center{Growth of frequently Birkhoff-universal functions of exponential type on rays} }
\author{Hans-Peter Beise}

\begin{document}
\maketitle

\begin{abstract}
We consider growth conditions for (frequently) Birkhoff-universal functions of exponential type with respect to the different rays emanating from the origin. For that purpose, we investigate their (conjugate) indicator diagram or, equivalently, their indicator function. Some known results, where growth is measured with respect to the maximum modulus, are extended. 
\end{abstract}

\section{Introduction and main Results}
Let $X$ be a topological vector space and $T$ a continuous operator on $X$. A vector $x\in X$ is called a\textbf{ hypercyclic vector} (for  $T$)
if its orbit $\{T^nx:n\in\N\}$ is dense in $X$. We denote by $\HC(T,X)$ the set of all hypercyclic vectors for $T$ on $X$. The operator $T$ is called \textbf{hypercyclic} (on $X$) if $\HC(T,X)\neq\emptyset$. In case that $X$ is an F-space, it is well-known that $\HC(T,X)$ is either empty or a dense $G_\delta$-set. In \cite{frequent}, F. Bayart and S. Grivaux introduced a stronger form of hypercyclicity, namely, the frequent hypercyclicity, as they impose a condition on the recurrence of the orbit of a single vector to each non-empty open set. A vector $x\in X$ is called a \textbf{frequently hypercyclic vector} (for $T$) if for each non-empty open set $U\subset X$ the set $\{n:T^nx\in U\}$ has positive lower density. By $\FHC(T,X)$ we denote the set of all frequently hypercyclic vectors for $T$ on $X$. The operator $T$ is said to be \textbf{frequently hypercyclic} if $\FHC(T,X)\neq\emptyset$. We recall that the \textbf{lower density} of a discrete set $\Lambda\subset \C$ is defined by
\[
 \liminf\limits_{r\rightarrow \infty}\frac{\#\{\lambda\in\Lambda:|\lambda|\leq r\}}{r}=:\ldens(\Lambda).
\]
Hypercyclicity is part of a more general concept, namely, universality. Here, a sequence of mappings $T_n:X\rightarrow Y$ ($n\in\N$) is considered where $X$ is some set and $Y$ is some topological space. The sequence $(T_n)$ is said to be \textbf{universal} if there exists some $x\in X$ such that $\{T_n x:n\in \N\}$ is dense in $Y$ and $x$ is called a \textbf{universal vector} in this case. In \cite{freqerd}, frequent hypercyclicity is generalised in the same way. The sequence $T_n:X\rightarrow Y$ ($n\in\N$) is said to be \textbf{frequently universal} if there exists some $x\in X$ such that for every non-empty open set $U\subset Y$ the set $\{n:T_nx\in U\}$ has positive lower density. \\

In this work, we are only concerned with spaces consisting of entire functions and with the translation operator $T_1$ that maps an entire function $f$ to $f(\cdot+1)$. A classical result in universality due to G. D. Birkhoff (cf. \cite{birkhoff}) states that $\HC(T_1,H(\C))\neq\emptyset$, where $H(\C)$ denotes the space of all entire functions endowed with the topology of uniform convergence on compact sets. It is also known that $\FHC(T_1,H(\C))\neq\emptyset$ (cf. \cite{frequent}). The elements in $\HC(T_1,H(\C))$ and $\FHC(T_1,H(\C))$ are called \textbf{Birkhoff-universal functions} and \textbf{frequently Birkhoff-universal functions}, respectively.\\
In \cite{duyos}, S. M. Duyos-Ruiz proves that Birkhoff-universal functions can have arbitrary slow transcendental rate of growth. Using the notation $M_f(r):=\max_{|z|=r}|f(z)|$, this means that for every $q:[0,\infty)\rightarrow [1,\infty)$ such that $q(r)\rightarrow \infty$, as $r$ tends to infinity, there are functions $f\in \HC(T_1,H(\C))$ such that $M_f(r)=O(r^{q(r)})$. This result is not valid for the case of frequent hypercyclicity. This is a consequence of results in \cite{frequenterd} that imply the following: There is no $f\in \FHC(T_1,H(\C))$ such that $M_f(r)=O(e^{\varepsilon r})$ for all $\varepsilon>0$. Conversely, for every fixed $\tau>0$, there exist functions $f\in \FHC(T_1,H(\C))$ such that $M_f(r)=O(e^{\tau r})$. In other words, there are no frequently Birkhoff-universal functions of exponential type zero, while for every positve $\tau>0$, there are frequently Birkhoff-universal functions of exponential type less or equal than $\tau$. 
Recall that an entire function $f$ is called a function of \textbf{exponential type $\tau$} if
\[
 \limsup\limits_{r\rightarrow\infty}\frac{\log M_f(r)}{r}=\tau
\]
and, $f$ is said to be a function of exponential type when the above $\tau$ is finite.\\
We extend known results by giving growth conditions for frequently Birkhoff-universal functions with respect to rays in the complex plane.\\
For an entire function $f$ of exponential type, the \textbf{indicator function} is defined by
\[
 h_f(\Theta):=\limsup\limits_{r\rightarrow\infty} \frac{\log |f(re^{i\Theta})|}{r},\ \ \Theta\in[-\pi,\pi].
\]
It is known that $h_f$ is determined by the support function of a certain compact and convex set $K(f)\subset\C$, more precisely, for $z=re^{i\Theta}$ we have
\begin{align}\label{indicator}
 r\, h_f(\Theta)=H_{K(f)}(z):=\sup\limits\{\Reee(zu):u\in K(f)\}
\end{align}
(cf. \cite[Theorem 1.3.21]{berenstein}). The set $K(f)$ is called the \textbf{conjugate indicator diagram} of $f$. For a given compact, convex set $K\subset \C$,
$\Exp(K)$ denotes the space of all entire functions $f$ of exponential type such that $K(f)\subset K$. The space $\Exp(K)$ is endowed with the topology induced by the norms
\[
 ||f||_{K,n}:=\sup\limits_{z\in\C}|f(z)|\,e^{-H_{K}(z)-\frac{1}{n}|z|},\ \ n\in\N
\]
which makes it to a Fr\'{e}chet space (cf. \cite{berenstein}, \cite{morimoto}). \\
According to the above, one observes that the functions $f\in \FHC(T_1,\Exp(K))$ are frequently Birkhoff-universal functions with restricted indicator function and thus satisfy certain growth conditions, specified by $K$, on the particular rays emanating from the origin.\\
We abbreviate $z\mapsto e^{\alpha z}$ by $e_\alpha$, where $\alpha$ is some complex number. From \cite[Theorem 5.4.12]{boas} we can conclude that for each function $f$ of exponential type, we have 
\begin{align}\label{ealphaind}
K(f e_\alpha)=\alpha+K(f)
\end{align}
since $K(e_\alpha)$ equals the singleton $\{\alpha\}$. In particular, for every polynomial $P$ we have $K(Pe_\alpha)=\{\alpha\}$.
Consequently, for a given non-empty, convex and compact set $K\subset\C$ and $\alpha\in K$ we have $\{Pe_\alpha:P\textnormal{ polynomial }\}\subset \Exp(K)$. By means of Runge's theorem, it is easily seen that $\{Pe_\alpha:P\textnormal{ polynomial }\}$ is dense in $H(\C)$ for each $\alpha\in \C$. This implies that $\Exp(K)$ is continuously and densely embedded in $H(\C)$ for every non-empty, compact and convex $K\subset\C$, and hence $\FHC(T_1,\Exp(K))\subset \FHC(T_1,H(\C))$. \\
For $v,w\in\C$ we denote by $[v,w]$ the closed line segment that connects $v$ and $w$.

\begin{theorem}\label{satzeins}
Let $v,w$ be complex numbers and $K:=[v,w]$. Then the set $\FHC(T_1,\Exp(K))$ is nonempty if and only if $v\neq w$ and $v,w\in i\R$, that is, $K$ is a non-singleton line segment of the imaginary axis.
\end{theorem}

In the following section, we discuss to what extent the necessary condition in this result extends to the frequent hypercyclicity in the weaker topology of $H(\C)$. Here, also more general sets $K$ are considered. 
\begin{remark}\label{bemerkeins}
If $K$ is a non-singleton, closed line segment of the imaginary axis, say $K:=[-ia,ia]+ib$ with real numbers $a,b$, where $a>0$, we have
\begin{align*}
h(\Theta):=H_K(e^{i\Theta}) =-b\sin(\Theta)+a\,|\sin(\Theta)|,\ \Theta\in[-\pi,\pi].
\end{align*}
Now, considering (\ref{indicator}), Theorem \ref{satzeins} implies the existence of functions \\$f\in\FHC(T_1,H(\C))$ of exponential type such that 
$h_f\leq h$.
\end{remark}
In the next two results we show that functions that belong to $\HC(T_1,\Exp(K))$ and $\FHC(T_1,\Exp(K))$ can have a very small rate of growth on the real axis provided $K$ contains two distinct points of the imaginary axis.
\begin{theorem}\label{reelklein}
Let $q:[0,\infty)\rightarrow [1,\infty)$ be such that $q(r)\rightarrow \infty$ as $r$ tends to infinity and $K\subset\C$ a compact, convex set containing two distinct points of the imaginary axis. Then there is a function $f\in\HC(T_1,\Exp(K))$ that satisfies
\[
|f(x)|=O(q(|x|))
\]
on the real axis.
\end{theorem}

\begin{theorem}\label{frequab}
Let $K\subset\C$ be a compact, convex set containing two distinct points of the imaginary axis and $c>1$. Then there is a function $f\in \FHC(T_1,\Exp(K))$ that satisfies
\[
|f(x)|=O(1+|x|^c)
\]
on the real axis.
\end{theorem}

\begin{remark}
By the corollary of \cite[Theorem 2.1]{hilbertentire}, which is an extension of the result of S. M. Duyos-Ruis mentioned above, we have $f\in \HC(T_1,\Exp(\{0\}))\neq \emptyset$. The assumptions of this corollary are satisfied since the polynomials are dense in $\Exp(\{0\})$ (see Proof of Theorem \ref{reelklein}). However, a function $f$ of exponential type zero (equivalently $f\in\Exp(\{0\})$) cannot satisfy $|f(x)|=O(x^k)$, $k\in \N$, on the real axis unless it is a polynomial of degree less or equal than $k$ (cf. \cite[Corollary 4.1.16]{berenstein}). Considering that $K(f)+\alpha=K(f e_\alpha)$, as outlined above, one verifies that no transcendental function of exponential type that has a singleton conjugate indicator diagram satisfies such a growth on the real axis. This shows that the assumptions in Theorem \ref{reelklein} and Theorem \ref{frequab} cannot be weakened to singleton sets $K$.
\end{remark}

\section{Further Results and Proofs}
In \cite{freqerd} and \cite{freqerderratum}, A. Bonilla and K.-G. Grosse-Erdmann extend a sufficient condition from \cite{frequent}, which is known as the frequent hypercyclicity criterion, to the case of frequent universality. We give it in the Fr\'{e}chet space version which is adequate for our application: A family of series $\sum_{n=1}^\infty x_{n,k}$, $k\in\N$, in some Fr\'{e}chet space $X$ is said to converge unconditionally, uniformly in $k$ if for every continuous seminorm $\rho$ on $X$ and every $\varepsilon>0$, there is some $N\in \N$ such that for any finite set $F\subset \N$ that satisfies $F\cap\{1,...,N\}=\emptyset$ and every $k\in\N$ we have $\rho(\sum_{n\in F}x_{n,k})<\varepsilon$.

\begin{theorem}[(Frequent Universality Criterion)]\label{freqcriterd}
Let $X$ be a Fr\'{e}chet space, $Y$ a separable Fr\'{e}chet space and $(T_n)$ a sequence of continuous operators from $X$ to $Y$. Assume that there are a dense subset $Y_0$ of $Y$ and mappings $S_n:Y_0\rightarrow X$, $n\in\N$, such that for all $y\in Y_0$:
\begin{enumerate}
\item[(1)] $\sum_{n=1}^k T_{k} S_{k-n} y$ converges unconditionally in $Y$, uniformly in $k\in\N$;
\item[(2)] $\sum_{n=1}^\infty T_{k} S_{k+n} y$ converges unconditionally in $Y$, uniformly in $k\in\N$;
\item[(3)] $\sum_{n=1}^\infty S_n y$ converges unconditionally in $X$;
\item[(4)] $T_n S_n y\rightarrow y$ as $n\rightarrow \infty$.
\end{enumerate}
Then $(T_n)$ is frequently universal.
\end{theorem}

This criterion leads to a general result that, as we will see later, contains Theorem  \ref{reelklein}, \ref{frequab} and the sufficient part of Theorem \ref{satzeins}.

\begin{theorem}\label{reelkleinsatz}
Let $q:[0,\infty)\rightarrow [1,\infty)$ be such that $q(r)\rightarrow \infty$ as $r$ tends to infinity and $d>0$. Then for every increasing sequence $(k_l)_{l\in\N}$ of positive integers that satisfies
\begin{align}\label{ildotti}
q(r)\geq l^c \textnormal{  for all  } r\in \left[(1-\delta)k_l,(1+\delta)k_l\right]=:I_{k_l}
\end{align}
for some $c>1$ and $\delta>0$,
there exists a function $f\in\Exp([-id,id])$ that is frequently universal for $(T_1^{k_l})_{l\in \N}$ on $\Exp([-id,id])$ and such that
\[
|f(x)|=O(q(|x|))
\]
on the real axis.
\end{theorem}

Our next aim is to find restriction for the conjugate indicator diagram of the functions that are of exponential type and belong to $\FHC(T_1,H(\C))$. Since the topolgy of $H(\C)$ is weaker than the topology of $\Exp(K)$, this restrictions extend to necessary conditions for compact and convex sets $K\subset \C$ to allow that $\FHC(T_1,\Exp(K))\neq\emptyset$.

\begin{theorem}\label{horizontal}
Let $v,w\in \C$ and $K:=[v,w]$. Further, if $v\neq w$, we assume that $K$ is a non-vertical line segment, that is, $\Reee(v)\neq \Reee(w)$. Then $\FHC(T_1,H(\C))\cap \Exp(K)=\emptyset$.
\end{theorem}

\begin{remark}
The above result implies that, in contrast to Remark \ref{bemerkeins}, there is no function $f\in \FHC(T_1,H(\C))$ that is of exponential type and such that 
$h_f(\Theta)=-b\sin(\Theta)+a\,|\sin(\Theta+\varepsilon)|$ with $0<\varepsilon<\pi$ and $a>0$, $b\in \R$.
\end{remark}

In the next result, we only consider functions of exponential type whose conjugate indicator diagram is contained in the left half-plane. Under this assumption, Theorem \ref{horizontal} is extended to more general sets.

\begin{theorem}\label{einpunktig}
Let $K$ be a compact, convex subset of $\{z:\Reee(z)\leq 0\}$ such that $K\cap i\R$ is a singleton. Then $\FHC(T_1,H(\C))\cap \Exp(K)=\emptyset$.
\end{theorem}

\begin{remark}\label{bemerkungHC}
In case that $K(f)\subset \{z:\Reee(z)<0\}$, the corresponding function $f$ tends to zero in the direction of the positive real axis and hence cannot be a Birkhoff-universal function (in the topology of $\Exp(K)$ or $H(\C)$). Consequently, $\HC(T_1, \Exp(K))=\emptyset$ whenever $K$ is strictly contained in the left half-plane. Since the hypercyclicity of $T_1$ on $\Exp(K)$ is equivalent to the hypercyclicity of its inverse $T_{-1}$ $(f\mapsto f(\cdot-1))$ on $\Exp(K)$ (cf. \cite[Corollary 1.3]{baymathbook}), we can conclude that $\HC(T_1,\Exp(K))$ is also empty in the case that $K$ is strictly contained in the right half-plane. However, we do not know if $\HC(T_1,H(\C))\cap \Exp(K)$ is non-empty for $K\subset\{z:\Reee(z)>0\}$. 
\end{remark}

Before giving the proofs to the above results, some notations and auxiliary facts are introduced.
For a compact set $K\subset \C$, we denote by $K^{-1}$ the set $\{z:1/z\in K\}\subset \C_\infty$, where $\C_\infty$ is the extended complex plane $\C\cup \{\infty\}$ and as usual $1/0:=\infty$. By $H_0(\C\setminus K)$ we mean the space of functions holomorphic on $\C\setminus K$ that vanish at infinity, and by $H(\C_\infty\setminus K^{-1})$ the space of functions holomorphic on $\C_\infty\setminus K^{-1}$ that vanish at infinity if $\infty\notin K^{-1}$. Both spaces are endowed with the topology of uniform convergence on compact sets. Recall that a function $F$ is holomorphic at infinity if $F(1/z)$ is holomorphic at the origin.
For a function $f$ of exponential type, $\Bo f(z):=\sum_{n=0}^\infty f^{(n)}(0)/z^{n+1}$ is called the \textbf{Borel transform} of $f$. The Borel transform is a holomorphic function on some neighbourhood of infinity that vanishes at infinity. It is known that $\Bo f$ admits an analytic continuation to $\C\setminus K(f)$ and further, given some compact, convex set $K\subset\C$,
\[
\Bo:=\Bo_K:\Exp(K)\rightarrow H_0(\C\setminus K),\ f\mapsto \Bo f|_{\C\setminus K}
\]
is an isomorphism (cf. \cite{morimoto}). For technical reasons, we consider $\TBo f(z):=1/z\,\Bo f(1/z)$ which defines a holomorphic function in $\C_\infty\setminus K(f)^{-1}$ that vanishes at infinity in case that $\infty\notin K(f)^{-1}$. Analogously to the Borel transform, for a given compact, convex set $K\subset \C$
\[
\TBo:=\TBo_K:\Exp(K)\rightarrow H_0(\C_\infty\setminus K^{-1}),\ f\mapsto \TBo f|_{\C_\infty\setminus K^{-1}}
\]
is an isomorphism.
The transform $\TBo$ is also used in \cite{muellerconv}.

\begin{lemma}\label{density}
Let $f$ be an entire function of exponential type. We further assume that $K$ is a compact, convex subset of $\C$ containing the origin and $A\subset\C$ is a bounded infinite set such that $f_\alpha(z):=f(\alpha z)$ belongs to $\Exp(K)$ for all $\alpha\in A$. Then $\linspan\{f_\alpha:\alpha\in A\}$ is dense in $\Exp(K)$ if and only if $f^{(k)}(0)\neq 0$ for every non-negative integer $k$.
\end{lemma}

\begin{proof}
We set $Y:=\linspan\{f_\alpha(z):\alpha\in A\}$.
Since the convergence in the topology of $\Exp(K)$ implies the convergence of the Taylor coefficients, the density of $Y$ ensures that the Taylor coefficients of $f$ do not vanish.\\ 
Let $\Lambda$ be an arbitrary element of the strong dual of $\Exp(K)$ that vanishes on $Y$. We prove that $\Lambda\equiv 0$, which implies the density of $Y$ by a standard application of the Hahn-Banach theorem.\\
By the Köthe duality (cf. \cite{koethedual}) and since $H_0(\C\setminus K)$ and $\Exp(K)$ are isomorphic, we can assume the existence of some $\omega$ holomorphic on a simply connected neighbourhood $\Omega_\omega$ of $K$ that represents $\Lambda$ by 
\begin{align}\label{phgeb}
\left\langle \Lambda,f\right\rangle=\frac{1}{2\pi i} \int_\Gamma \omega(\xi)\, \Bo f(\xi) \,d\xi
\end{align}
where $\Gamma$ is a closed curve around $K$ with positive orientation with respect to the points in $K$.
We show the existence of some connected open set $D$ containing $A\cup\{0\}$ and such that 
\begin{align}\label{pbgeb}
K(f_\alpha) \subset \Omega_\omega \textnormal{ for all } \alpha\in D.
\end{align}
A simple observation using the equality $H_{K(f)}(\xi)=h_f(\arg(\xi))$, $\xi\in \{z:|z|=1\}$, yields $K(f_\alpha)\subset\alpha\, K(f)$ for arbitrary $\alpha\in\C$. Hence, the condition $f_\alpha\in \Exp(K)$ implies $A\, K(f)\subset K$. Since $0\in K$ and $K$ is convex, we obtain $r\,\alpha\, K(f)\subset K$ for all $r\in[0,1]$ and all $\alpha\in A$.  The continuity of the multiplication and the compactness of $K(f)$ imply the existence of some open neighbourhood $U(\alpha,r)$ for each pair $(r,\,\alpha)\in [0,1]\times A$ such that $U(\alpha,r)\,K(f)\subset \Omega_\omega$. Now $D:=\bigcup_{\alpha\in A}\bigcup_{r\in [0,1]} U(\alpha, r)$ has the claimed properties of (\ref{pbgeb}).\\
Since $A$ is assumed to be bounded and infinite, it has an accumulation point in $D$.
Considering the Taylor expansion $\TBo f(z):=\sum_{n=0}^\infty f^{(n)}(0) z^n$, which converges on some neighbourhood of the origin, we immediately obtain 
\begin{align*}
\Bo f_\alpha(\xi) =\frac{\TBo f(\frac{\alpha}{\xi})}{\xi}
\end{align*}
for all $\xi\in \C\setminus K$ and $\alpha \in D$. Now, with (\ref{phgeb}),  
\[
a(\alpha):=\left\langle \Lambda,f_\alpha\right\rangle=\frac{1}{2\pi i} \int_\Gamma \omega(\xi)\, \Bo f_\alpha (\xi) \,d\xi=\frac{1}{2\pi i} \int_\Gamma \frac{\omega(\xi)}{\xi}\,\TBo f\left(\frac{\alpha}{\xi}\right)\, d\xi
\]
defines a holomorphic function on $D$. The condition $\Lambda|_Y\equiv 0$ means that $a$ vanishes on $A$ and, since $A$ has an accumulation point, the identity theorem yields $a\equiv0$. Consequently,
\[
a^{(k)}(\alpha)=\frac{1}{2\pi i} \int_\Gamma \frac{\omega(\xi)}{\xi^{k+1}}\,(\TBo f)^{(k)}\left(\frac{\alpha}{ \xi}\right) \,d\xi=0
\]
for all $\alpha\in D$ and every $k\in\N\cup\{0\}$. With $\alpha=0$ this means 
\begin{align}\label{kldo}
(\TBo f)^{(k)}(0)\,\frac{1}{2\pi i}\, \int_\Gamma \frac{\omega(\xi)}{\xi^{k+1}} \,d\xi=0
\end{align}
for all $k\in \N\cup\{0\}$. The assumption $f^{(k)}(0)\neq 0$ for all $k\in \N\cup \{0\}$ is equivalent to $(\TBo f)^{(k)}(0)\neq 0$ for all $k\in \N\cup \{0\}$ as one can immediately see in the Taylor expansion of $\TBo f$. Consequently, equality (\ref{kldo}) shows $\omega\equiv 0$ and this means $\Lambda\equiv 0$.
\end{proof}

\begin{proof}[Proof of Theorem \ref{reelkleinsatz}]
We set $K:=[-id,id]$
and $Y_0=\linspan\left\{f_\alpha: \alpha\in [-i\frac{d}{2},i\frac{d}{2}]\right\}$ where 
\[f_\alpha(z):=\frac{(e^{\alpha z} -1)^2}{z^2}.
\]
Then $Y_0$ is dense in $\Exp(K)$ by Lemma \ref{density}, and $K(f_\alpha)\subset K$. Furthermore, let
\[
E:=\left\{f\in \Exp(K):\sup\limits_{x\in \R} \frac{|f(x)|}{q(|x|)} <\infty\right\}
\] 
be endowed with the topology induced by the norms $(||\cdot||_n)_{n\in \N}$ defined by
\[
||f||_{n}:= ||f||_{K,n}+ \sup\limits_{x\in\R} \frac{|f(x)|}{q(|x|)}.
\]
Then $E$ is a Fréchet space. \\
We apply the frequent universality criterion to the mappings
\[
T_{1}^{k_l}:E\rightarrow \Exp(K)
\]
and 
\[S_{k_l}:\Exp(K)\rightarrow E
\]
with $S_{k_l}f:=f(\cdot-k_l)$. \\
In a first step, it is shown that in condition (1) and (2) of this criterion, we have absolute convergence for $(k_l)_{l\in\N}=(k)_{k\in\N}$. 
Let $n\in\N$ and $\alpha\in [-i\frac{d}{2},i\frac{d}{2}]$ be fixed. Taking into account that $\alpha$ is purely imaginary, we obtain 
\begin{align}\label{tigerente}
&\left| \left(e^{\alpha (z+k)}-1\right)^2\right|\, e^{-H_K(z)-\frac{1}{2n}|z|}\leq \left(\left|e^{2\alpha z}\right|+2\left|e^{\alpha z}\right|+1\right)\,e^{-H_K(z)-\frac{1}{2n}|z|}<\infty
\end{align}
for all non-negative integers $k$. Hence 
\begin{align}\label{kinzig}
\sup\limits_{|z+k|>\frac{k}{2}}& |f_\alpha(z+k)|\,e^{-H_K(z)-\frac{1}{n}|z|}\notag \\
&\leq \sup\limits_{|z+k|>\frac{k}{2}}\frac{|e^{\alpha(z+k)}-1|^2}{\left(\frac{k}{2}\right)^2}\,e^{-H_K(z)-\frac{1}{n}|z|}= O\left( \frac{1}{k^2}\right).
\end{align}
Again from inequality (\ref{tigerente}), it follows that
\begin{align}\label{deraltekinzig}
\sup\limits_{1\leq|z+k|\leq \frac{k}{2}}& |f_\alpha(z+k)|\,e^{-H_K(z)-\frac{1}{n}|z|}\notag \\
&\leq \sup\limits_{1\leq|z+k|\leq \frac{k}{2}} |e^{\alpha (z+k)}-1|^2\,e^{-H_K(z)-\frac{1}{2n}|z|}\,e^{-\frac{1}{2n}|z|}=O\left( e^{-\frac{k}{4\,n}}\right).
\end{align}
Obviously,
\begin{align}\label{tigerentenbrille}
\sup\limits_{|z+k|\leq1} |f_\alpha(z+k)|\,e^{-H_K(z)-\frac{1}{n}|z|}\leq \sup\limits_{|z|\leq 1} |f_\alpha(z)|\, e^{-\frac{k-1}{n}}=O\left( e^{-\frac{k-1}{n}}\right).
\end{align}
The estimations (\ref{kinzig}), (\ref{deraltekinzig}) and (\ref{tigerentenbrille}) now yield
\[
||T_{1}^k f_\alpha||_{K,n}=\sup\limits_{z\in\C} \left|\frac{(e^{\alpha (z+k)} -1)^2}{(z+k)^2}\right|\,e^{-H_K(z)-\frac{1}{n}|z|}=O\left( \frac{1}{k^2}+ e^{-\frac{k}{4\,n}}+ e^{-\frac{k-1}{n}}\right).
\]
Thus, the series $\sum_{k=1}^\infty ||T_{1}^k f_\alpha||_{K,n}$ converges for every $n\in\N$. Since the convergence remains for every function $f\in Y_0$ and observing that $\sum_{k=1}^m T_1^mS_{m-k} f=\sum_{k=1}^mT_1^kf$, this implies that condition (1) of the frequent universality criterion is satisfied. With the same reasoning, one obtains the unconditional convergence of $\sum_{k=1}^\infty || S_{k} f||_{K,n}$ in $\Exp(K)$ for all $f\in Y_0$ and all $n\in\N$. This yields condition (2) of the frequent universality criterion. Obviously, a similar reasoning holds for all subsequences of $(k)_{k\in\N}$. 
As condition (4) is obvious, there is only (3) left to show. 
On the real axis, the moduli of the functions $f_\alpha$ are bounded. With assumption (\ref{ildotti}), we obtain  
\begin{align}\label{lldotti}
\sup\limits_{x\in I_{k_l}\cup -I_{k_l}} \frac{|f_\alpha(x-k_l)|}{q(|x|)}=O\left( \frac{1}{l^c}\right).
\end{align}
If $x\notin I_{k_l}\cup -I_{k_l}$, then $|x-k_l|>\delta k_l$ and hence
\begin{align}\label{ausduennung}
\sup\limits_{x\in \R\setminus( I_{k_l}\cup-I_{k_l})}\frac{|f_\alpha(x-k_l)|}{q(|x|)}\leq \sup\limits_{x\in \R\setminus( I_{k_l}\cup-I_{k_l})}\frac{|e^{\alpha (x-k_l)}-1|^2}{(\delta k_l)^2}=O\left(\frac{1}{k_l^2}\right).
\end{align} 
Combining (\ref{lldotti}) and (\ref{ausduennung}), we have
\[
\sum_{l=1}^\infty \sup\limits_{x\in \R} \frac{|f_\alpha(x-k_l)|}{q(|x|)}<\infty.
\]
Together with the convergence of $\sum_{l=1}^\infty ||S_{k_l} f_\alpha||_{K,n}$ for each $n\in\N$, which is shown in the first part of the proof, this implies the convergence of $\sum_{l=1}^\infty ||S_{k_l} f_\alpha||_{n}$ for every $n\in \N$. We obtain the unconditional convergence of $\sum_{l=1}^\infty S_{k_l} f$ in $E$ for every $f\in Y_0$. Now, the frequent universality criterion yields that $(T_{1}^{k_l})_{l\in\N}$ is frequently universal from $E$ to $\Exp(K)$. 
\end{proof}

The Carleman formula (cf. \cite[page 45]{rubel}) states that, for a function $f$ holomorphic on an open superset of $\{z:\Reee(z)\geq 0\}$ and having the zeros $(\lambda_n=r_n e^{i\Theta_n})$ in $\{z:\Reee(z)>0\}$ and none on the imaginary axis, the following holds:
\begin{align}\label{christbaum}
\sum_{r_n\leq R} \left(\frac{1}{r_n}-\frac{r_n}{R^2}\right)\,\cos(\Theta_n)=&\frac{1}{2\pi}\int_0^R\left(\frac{1}{t^2}-\frac{1}{R^2}\right)\,\log|f(it)f(-it)|\,dt\,+\notag\\[2mm]
&\frac{1}{\pi R}\int_{-\pi/2}^{\pi/2}\log|f(Re^{i\Theta})|\,\cos(\Theta)\, d\Theta+O(1)
\end{align}
for $R>0$ such that no zero of $f$ is located on $\{z:|z|=R\}\cap \{z:\Reee(z)\geq 0\}$.
We apply this formula to derive the following
\begin{lemma}\label{nullenlemma}
Let $f$ be an entire function of exponential type and
\[c:=\frac{1}{2}\,\max\{|\Imm(z)-\Imm(w)|:z,w\in K(f)\}=\frac{1}{2}\,|h_f(\pi/2)+h_f(-\pi/2)|.
\]
We further assume that $\lambda=(r_n\,e^{i \Theta_n})\subset \{z:|\arg(z)|<\gamma\}$, $0<\gamma<\pi/2$, is a sequence of zeros for $f$. Then, if $f$ is not constantly zero, 
\[\ldens (\lambda)\leq \frac{c}{\pi \,\cos(\gamma)}.
\]  
\end{lemma}

\begin{proof}
We set $n(R):=\#\{\lambda_n\in\lambda:|\lambda_n|\leq R\}$ and assume that $\lambda$ are the only zeros of $f$ in the right half-plane. Otherwise, the effect which leads to the  assertion is strengthened. Since $f$ is of exponential type, we have
\begin{align}\label{weihnachtsbaum}
n(R)=O(R)
\end{align}
(cf. \cite[Theorem 2.5.13]{boas}). Integration by parts yields
\[
\sum_{r_n\leq R} \frac{1}{r_n}=\int_0^R \frac{1}{t}\, dn(t)=\int_{0}^R \frac{n(t)}{t^2}\,dt+\frac{n(R)}{R}
\]
and hence, considering (\ref{weihnachtsbaum}), we obtain that
\begin{align}\label{weihnachtsbaum1}
 \sum_{r_n\leq R} \left(\frac{1}{r_n}-\frac{r_n}{R^2}\right)\,\cos(\Theta_n)\geq \cos(\gamma)(\ldens(\lambda)-\varepsilon_1)\log(R)\,+\,O(1)
\end{align}
for all $\varepsilon_1>0$. From $\log M_f(R)=O(R)$ it immediately follows that
\begin{align}
 \frac{1}{\pi R}\int_{-\pi/2}^{\pi/2}\log|f(Re^{i\Theta})|\,\cos(\Theta)\, d\Theta=O(1)
\end{align}
and, by the definition of $c$,
\begin{align}\label{weihnachtsbaum2}
 \frac{1}{2\pi}\int_0^R\left(\frac{1}{t^2}-\frac{1}{R^2}\right)\,\log|f(it)f(-it)|\,dt\leq \frac{c+\varepsilon_2}{\pi}\,\log(R)\,+\, O(1)
\end{align}
for all $\varepsilon_2>0$.
Now, inserting (\ref{weihnachtsbaum}), (\ref{weihnachtsbaum1}) and (\ref{weihnachtsbaum2}) in (\ref{christbaum}) leads to
\[
(\ldens(\lambda)-\varepsilon_1)\log(R)\leq \frac{c+\varepsilon_2}{\pi\,\cos(\gamma)}\,\log(R)\,+\, O(1) \ \ \textnormal{for all }\varepsilon_1,\varepsilon_2>0,
\] 
and this implies the assertion.
\end{proof}

\begin{proof}[Proof of Theorem \ref{horizontal}]
Suppose that $f\in \FHC(T_1,H(\C))$ is of exponential type with $K(f)=[v,w]$. Then we can find a sequence $(\lambda_n)$ of positive integers that has positive lower density and such that 
\begin{align*}
\sup\limits_{|z|\leq \frac{1}{2}} |f(z+\lambda_n)-z|<\frac{1}{2}
\end{align*}
for all $n\in\N$. Rouché's theorem implies that $f$ has a zero $\tilde{\lambda}_n$ in $\{z:|z-\lambda_n|<\frac{1}{2}\}$ for every positive integer $n$. Obviously, also $\ldens((\tilde{\lambda}_n))=\ldens((\lambda_n))>0$.\\ 
Without loss of generality, let $K(f)=\{re^{i\Theta}:r\in [0,a]\}$ with $0\leq |\Theta|<\frac{\pi}{2}$ and some $a\in[0,\infty)$. 
Otherwise, consider $e_\alpha f$ with an appropriate $\alpha\in\C$ such that $K(e_\alpha f)=\alpha +K(f)$ has the desired position.
The rotation of the arguments of $f$ by an angle of $-\Theta$ causes the rotation of the conjugate indicator diagram by the same angle. This can be easily deduced from the connection between the indicator function and the conjugate indicator diagram. That means, for $g(z):=f(e^{-i\Theta}z)$ we have $K(g)=e^{-i\Theta} K(f)$. Thus $K(g)$ is a horizontal line segment and consequently, 
\begin{align}\label{oemchen}
\max\{|\Imm(z)-\Imm(w)|:z,w\in K(g)\}=0. 
\end{align}
Further, $(e^{i\Theta}\tilde{\lambda}_n)$ is a sequence of zeros of $g$.
Since $|\Theta|<\frac{\pi}{2}$, we can find some $\gamma\in (|\Theta|,\frac{\pi}{2})$ such that $(e^{i\Theta}\tilde{\lambda}_n)_{n\geq n_0}$ is contained in the sector $\{z:\arg(z)<\gamma\}$ for sufficiently large $n_0$. According to the lower density of $(\tilde{\lambda}_n)$ and (\ref{oemchen}), this is a contradiction to Lemma \ref{nullenlemma}.
\end{proof}

\begin{proof}[Proof of Theorem \ref{einpunktig}]
Assuming the contrary, we suppose that $f\in \FHC(T_1,H(\C))$ is a function of exponential type such that $K(f)=K$. 
Then there exists a sequence $(\lambda_n)$ of positive integers with $\ldens((\lambda_n))>0$ and such that 
\begin{align}\label{stubbis}
\max_{|z|\leq \frac{1}{2}} |f(z+\lambda_n)-z|<\frac{1}{4}.
\end{align} 
Let $K_1:=K(f)\cap\{z:\Reee(z)\geq -\varepsilon\}$ with $\varepsilon>0$ so small that 
\begin{align}\label{jonnyg}
\max\{|\Imm(z)-\Imm(w)|: z,w\in K_1\}<2\pi \,\ldens((\lambda_n)).
\end{align}
This is possible since the intersection of $K(f)$ with the imaginary axis is a singleton. Due to Aronszajn's theorem (cf. \cite{aron}), we can decompose $\Bo f$ into the sum 
\[
H_1+H_2=\Bo f
\]
where $H_1\in H_0(\C\setminus K_1)$ and $H_2\in H_0(\C\setminus K_2)$ with $K_2:=\overline{K(f)\setminus K_1}$. We set $h_1=\Bo^{-1}_{K_1}(H_1)$ and $h_2=\Bo^{-1}_{K_2}(H_2)$. The set $K_2$ is strictly contained in the left half-plane and thus there are some $\tau>0$ and $C>0$ such that 
\[
|h_2(z)|< C e^{-\tau|z|}
\]
in some sector $\{z:|\arg(z)\leq \varphi\}$ with $\varphi>0$ sufficiently small. Now, (\ref{stubbis}) yields
\begin{align}\label{stubbistubant}
\frac{1}{4}&>\max_{|z|\leq \frac{1}{2} } |f(z+\lambda_n)-z|=\max_{|z|\leq \frac{1}{2} } |h_1(z+\lambda_n)+h_2(z+\lambda_n)-z|\notag\\
&>\max_{|z|\leq \frac{1}{2} } |h_1(z+\lambda_n)-z|-Ce^{-\tau (\lambda_n-\frac{1}{2})}
\end{align}
for all $n\in\N$. Since $Ce^{-\tau (\lambda_n-\frac{1}{2})}\rightarrow 0$ as $n$ tends to infinity, (\ref{stubbistubant}) implies
\[
\max\limits_{|z|\leq \frac{1}{2} }|h_1(z+\lambda_n)- z|<\frac{1}{2}
\]
for $n$ larger than some $n_0\in\N$. Then, by Rouché's theorem, $h_1$ has a zero $\tilde{\lambda}_n$ in $\{z:|z-\lambda_n|\leq \frac{1}{2}\}$ for all $n>n_0$. The sequence $(\tilde{\lambda}_n)_{n>n_0}$ has the same lower density as $(\lambda_n)$. By (\ref{jonnyg}), the conjugate indicator diagram $K(h_1)$ has an extension in the direction of the imaginary axis less than $2\pi\, \ldens((\lambda_n))$, and for all $\gamma>0$ the sequence $(\tilde{\lambda}_n)_{n>n_0}$ is contained in the sector $\{z:|\arg(z)|<\gamma\}$ if $n_0$ is sufficiently large. This contradicts Lemma \ref{nullenlemma}.
\end{proof}

Now, by means of the shown results, we can proceed with the proofs of the Theorems of the first section.
\begin{proof}[Proof of Theorem \ref{reelklein}]
By (\ref{ealphaind}), $\{Pe_\alpha: P\textnormal{ polynomial }\}\subset \Exp(K)$ for every $\alpha\in K$.
We show that, for each $\alpha\in K$,  $\{Pe_\alpha: P\textnormal{ polynomial }\}$ is dense in $\Exp(K)$.
Let $\Sigma$ denote the space of all polynomials.\\
In a first case we assume that $0\in K$. 
Then $\Sigma$ is dense in $H(\C_\infty\setminus K^{-1})$ by Runge's theorem. Observing that $\TBo^{-1}(\Sigma)=\Sigma$, this shows that $\Sigma$ is dense in $\Exp(K)$ since $\TBo$ is an isomorphism.\\
Let $K$ be an arbitrary compact and convex set. Considering (\ref{ealphaind}), we have $g=f/e_\alpha\in \Exp(K-\{\alpha\})$ for arbitrary $f\in \Exp(K)$ and $\alpha\in K$. Further,   
\begin{align*}
 ||f||_{K,n}&=\sup_{z\in\C} |g(z)||e^{\alpha z}|e^{-H_K(z)-\frac{1}{n}|z|}\\
 &=\sup_{z\in\C} |g(z)|e^{-H_K(z)-H_{\{-\alpha\}}(z)-\frac{1}{n}|z|} \\&=\sup_{z\in\C} |g(z)|e^{-H_{K-\{\alpha\}}(z)-\frac{1}{n}|z|}\\
 &=||g||_{K-\{\alpha\},n}
\end{align*}
shows that $f\mapsto f/e_\alpha$ is an isometric isomorphism from $\Exp(K)$ to $\Exp(K-\{\alpha\})$. With the first part, this implies the assertion.
Now, we can conclude that for two compact, convex sets $K\subset L$, the space $\Exp(K)$ is continuously and densely embedded in $\Exp(L)$ and thus $\FHC(T_1,\Exp(K))\subset \FHC(T_1,\Exp(L))$.\\
Let $ia,ib\in K$ with real numbers $a<b$.
According to the above and taking into account that the convexity of $K$ implies $[ia,ib]\subset K$, it is sufficient to consider $\Exp([ia,ib])$. 
We choose $0<d<b-a$.
Since $q$ tends to infinity, we can find an increasing sequence $(k_l)_{l\in\N}$ of positive integers such that condition (\ref{ildotti}) holds. Now, Theorem \ref{reelkleinsatz} provides a function $g\in\Exp([-id,id])$ that is frequently universal for $(T_{1}^{k_l})_{l\in\N}$ on $\Exp([-id,id])$ and such that $|g(x)|=O(q(|x|))$ on $\R$. In particular, this implies that $g\in\HC(T_1,\Exp([-id,id]))$. We choose a suitable $\beta\in \Q$ such that $f:=e_{2\pi i \beta}\,g\in \Exp([ia,ib])$ (cf. (\ref{ealphaind})). Note that $f$ satisfies $
|f(x)|=O(q(|x|))$
on the real axis. There exists an $m\in\N$ such that $\beta\,k\,m\in \N$, and hence $e_{2\pi i \beta}(z+k\,m)=e^{2\pi i\beta(z+k\, m)}=e_{2\pi i \beta}(z)$ for all $k\in \N$. We obtain that 
\[(T_{1}^m)^k f =e_{2\pi i\beta}\, (T_{1}^m)^k g
\]
holds for all $k\in \N$ and, since $g\in \HC(T_{1}^m,\Exp([-id,id]))$ due to Ansari's theorem (cf. \cite[Theorem 3.1]{baymathbook} and \cite{ansari}), one easily verifies that $f$ is the desired function.
\end{proof}

\begin{proof}[Proof of Theorem \ref{frequab}]
As in the proof of Theorem \ref{reelklein}, we can consider $\Exp([ia,ib])$ with $ia,ib\in K$ and real numbers $a<b$.
With $q(r):=1+r^c$, $r\geq 0$, condition (\ref{ildotti}) is satisfied for the whole sequence $(k_l)_{l\in\N}=(k)_{k\in\N}$. Consequently, for arbitrary $d>0$, Theorem \ref{reelkleinsatz} provides a function $g\in\FHC(T_1, \Exp([-id,id]))$ satisfying the desired growth condition on $\R$. Now, for each $0<d<b-a$ and a suitable choice of $\beta\in\Q$ we have that $f:=e_{2\pi i \beta}\,g\in\Exp([ia,ib])$ (cf. (\ref{ealphaind})). Taking into account that Ansari's theorem still holds for frequent hypercyclicity (cf. \cite[Theorem 6.30]{baymathbook}), a similar reasoning as in the proof of Theorem \ref{reelklein} shows that $f\in \FHC(T_1,\Exp([ia,ib]))$. 
\end{proof}

\begin{proof}[Proof of Theorem \ref{satzeins}]
In case that $[v,w]$ is a non-singleton line segment of the imaginary axis, Theorem \ref{frequab} yields the assertion.
By the reasoning in Remark \ref{bemerkungHC}, we have $\HC(T_1,\Exp([v,w]))=\emptyset$ and hence $\FHC(T_1,\Exp([v,w]))=\emptyset$
if $[v,w]$ is strictly contained in the left or right half-plane. In the remaining case that $[v,w]$ is a singleton or non-vertical line segment intersecting the imaginary axis, we obtain $\FHC(T_1,\Exp([v,w]))=\emptyset$ by Theorem \ref{horizontal}.
\end{proof}

\bibliographystyle{plain}                      
\bibliography{references}

\end{document}